\documentclass[12pt]{amsart}
\usepackage[centertags]{amsmath}
\usepackage{amsfonts}
\usepackage{amssymb}
\usepackage[latin5]{inputenc}
\usepackage{amsthm}
\usepackage{mathrsfs}
\usepackage{upgreek}
\usepackage{amsopn}
\usepackage{amscd}
\newtheorem{theorem}{Theorem}[section]
\newtheorem{lemma}[theorem]{Lemma}
\newtheorem{definition}[theorem]{Definition}

\newtheorem{prop}[theorem]{Proposition}
\theoremstyle{definition}

\newtheorem{rem}[theorem]{Remark}

\theoremstyle{remark}

\numberwithin{equation}{section}

\renewcommand{\leq}{\leqslant}
\renewcommand{\geq}{\geqslant}

\renewcommand{\epsilon}{\varepsilon}

\begin{document}
{\setlength{\baselineskip}%
                        {1.3\baselineskip}
                      
\title{Towards a Theory of Unbounded Locally Solid Riesz Spaces}
\author{Z.ERCAN, M.VURAL}
\address{(Zafer Ercan) Department of Mathematics, Abant \.{I}zzet Baysal University, G\"{o}lk\"{o}y Kamp\"{u}s\"{u}, 14280 Bolu, Turkey}
\address{(Mehmet Vural) Department of Mathematics, Abant \.{I}zzet Baysal University, G\"{o}lk\"{o}y Kamp\"{u}s\"{u}, 14280 Bolu, Turkey}
\email{zercan@ibu.edu.tr , m.vural.hty@gmail.com}

\maketitle 
\begin{abstract} We introduce the notion of unbounded locally solid Riesz spaces, and investigate its fundamental properties.   
\end{abstract}
\section{Introduction} One of the main convergence types in a Riesz space \footnote{In this paper all Riesz spaces will be assumed Archimedean} is {\bf\emph{order convergence}}. Recall that a net $(x_{\alpha})_{\alpha \in A}$ in a Riesz space $E$ is said to be order convergent to $x\in E$ (briefly; $x_{\alpha}\xrightarrow{o} x$ or $x_{\alpha}$ $o$-converges to $x$) if there exists  another net $(y_{\beta})_{\beta \in B}$ in $E$ such that

i) $y_{\beta}\downarrow 0$, that is, $(y_{\beta})_{\beta \in B}$ is decreasing to 0 and ;

ii) For each $ \beta \in B$ there exists ${\alpha}_{0} \in A$ such that $|x_{\alpha}-x|\leq y_{\beta}$ for each ${\alpha}\geq {\alpha}_{0}$.

Unbounded order convergence in a Riesz space was defined and studied in \cite{nakano} and \cite{wickstead}.
Recently, many authors have started to work on this topic in \cite{gtx},\cite{gx} and \cite{Gao}. Namely, a net $(x_{\alpha})$ \footnote{The index is not written unless it is necessary} in a Riesz space $E$ is said to be {\bf\emph{unbounded order convergent}}  if the net $(|x_{\alpha}-x|\wedge u)$ is order convergent to zero for each $u\in E^{+}$ (briefly; $x_{\alpha}\xrightarrow{uo} x$ or $x_{\alpha}$ $uo$-converges to $x$). In general, every order convergent net is unbounded order convergent but the converse  is not true (For example, consider $c_{0}$ as a Riesz space under pointwise order, the standard unit vectors $(e_{n})$ are uo-convergent but not o-convergent). It is obvious that order convergence and unbounded order convergence coincide for order bounded nets. Although in general, unbounded order convergence is not topological  (see \cite{gtx}), but in an atomic Riesz space, it is topological  (see Theorem 2 in \cite{dem}).

Let $E$ be a normed Riesz space. A net $(x_{\alpha})$ in $E$ is said to be {\bf\emph{unbounded norm convergent}} to $x\in E$ if 
the net $(|x_{\alpha}-x|\wedge u)$ is norm convergent to zero for each $u\in E^{+}$ (briefly; $x_{\alpha}\xrightarrow{un} x$ or $x_{\alpha}$ $un$-converges to $x$). The notion of unbounded norm convergence was defined in \cite{tro} and many results 
were obtained in \cite{dot} and \cite{kmt}. These notions were extended to the locally solid Riesz spaces (see \cite{dem1}). In \cite{kmt}, it  is noticed that unbounded norm convergence defines a topology, that is, there exists a new topology on the normed Riesz space $E$ so that the unbounded norm convergence and topological convergence  agree. This new topology is called {\bf\emph{un-topology}} . See \cite{kmt} for more details on this topic. In the paper, it is also proved that in Banach lattices, the norm topology and unbounded norm topology  coincide if and only if the space has a strong order unit (Theorem 2.3).
\section{An observation} Let $(E,||.||)$ be a normed Riesz space. For each $u\in E^{+}$, define $P_{u}:E\rightarrow\mathbb{R}$ by  
\begin{center}$P_{u}(x)=|||x|\wedge u||.$\end{center} We will show that for each $u\in E^{+}$, $P_{u}$ is a Riesz pseudonorm. Recall that a real-valued map $p$ on $E$ is called a {\bf\emph{Riesz pseudonorm}} \cite{frem} if the following conditions are satisfied:
\begin{enumerate}
\item[1.] $p(x)\geq 0$  $\forall x \in X$;
\item[2.] If $x=0$, then $p(x)=0$;
\item[3.] $p(x+y)\leq p(x)+p(y)$ $\forall x,y \in X$;
\item[4.] If $\lim_{n \rightarrow \infty} \lambda_{n}=0 $ in $\mathbb{R}$, then $p(\lambda_{n} x)\rightarrow 0 $ in $\mathbb{R}$   $\forall x \in X$;
\item[5.] If $|x| \leq |y|$, then $p(x) \leq p(y)$.
\end{enumerate}
\begin{theorem} Let $(E,||.||)$ be a normed Riesz space. For each $u\in E^{+}$, the map $P_{u}:E\longrightarrow\mathbb{R}^{+}$ defined by  $P_{u}(x)=|||x|\wedge u||$, is a Riesz pseudonorm. Moreover, the un-topology and the topology generated by the family $(P_{u})_{u\in E^{+}}$ are coincide.   
\end{theorem}
\begin{proof} Let $u\in E^{+}$ be given. Obviously, the conditions (1),(2) and (5) hold. For condition (3): Let  $x$, $y\in E$ be given. Since $|x+y|\leq  |x|+|y|$, we have 
$|x+y|\wedge u\leq (|x|+|y|)\wedge u\leq |x|\wedge u+|y|\wedge u$ and since $||.||$ is a lattice norm, we get the inequality $P_{u}(x+y)\leq P_{u}(x)+P_{u}(y)$ by the monotonicity and the triangle inequality property of lattice norm. For condition(4): Let $\{{\lambda}_{n}\} \subset\mathbb{R}$ be a sequence such that $\lim_{n \rightarrow \infty} \lambda_{n}=0 $ and $x\in E$, the inequality \begin{center} $P_{u}({\lambda}_{n}x)=|||{\lambda}_{n}x|\wedge u|| \leq ||{\lambda}_{n}x||=|{\lambda}_{n}|||x||$ \end{center} implies that
$\lim_{n \rightarrow \infty}P_{u}(\lambda_{n}x)=0$. Hence $P_{u}$ is a Riesz pseudonorm.

Let $(x_{\alpha})$ be a net converging to $x$ in $un$-topology, that is, $||~ |x_{\alpha}-x|\wedge u~||\rightarrow 0$ for each $u\in E^{+}$. By definition, $P_{u}(x_{\alpha}-x)$ converges to zero for each $u\in E^{+}$, so it converges to $x$ in the topology generated by  the family $(P_{u})_{u\in E^{+}}$. Converse direction is also true. This completes the proof.  
\end{proof}
Note that in a Riesz space, we have the following.
\begin{theorem} Let $E$ be a Riesz space and $p:E\rightarrow\mathbb{R}$ be a map. The followings are equivalent:
\begin{enumerate}
\item[i.] $p$ is a Riesz pseudonorm;
\item[ii.] $p(x)=p(|x|)$ for all $x\in E$ and for each $u\in E^{+}$, the map $p_{u}:E\rightarrow\mathbb{R}$, defined by $p_{u}(x)=p(|x|\wedge u)$, is a  Riesz pseudonorm.
\end{enumerate}
\end{theorem}
\begin{proof} 
 If $(i)$ holds, following the proof of  Theorem(2.1), we can get $(ii)$. Suppose that $(ii)$ holds. Since $p(x)=p(|x|)=p_{|x|}(x)\geq 0$, it is obvious that $p(x)=p(|x|)=p_{|x|}(x)=0$ whenever $x=0$.
Let $x$, $y\in E$ be given. Then
\begin{center}$p(x+y)=p(|x+y|\wedge (|x|+|y|)$\end{center}\begin{center}$\hspace{0.5 cm}$ $=p_{|x|+|y|}(|x+y|)$\end{center} \begin{center}$\hspace{2.1 cm}$ $\leq p_{|x|+|y|}(|x|)+p_{|x|+|y|}(|y|)$\end{center}\begin{center}$\hspace{4.6 cm}$ $=p(|x|\wedge (|x|+|y|))+p(|x|\wedge (|x|+|y|)$\end{center} \begin{center}$\hspace{0.6 cm}$ $=p(|x|)+p(|y|)$\end{center}\begin{center}$=p(x)+p(y)$ \end{center} so that $p$ satisfies the triangle inequality. Let $x\in E$ be given. Then
\begin{center}$\lim_{n\rightarrow \infty}p({\lambda}_{n}x)= \lim_{n \rightarrow \infty}p(|{\lambda}_{n}x|)= \lim_{n\rightarrow \infty}p_{|x|}(|{\lambda}_{n}||x|)=0$.\end{center} If $|x|\leq |y|$ then
\begin{center}$p(x)=p(|x|)=p(|x|\wedge |y|)=p_{|y|}(|x|)=p_{|y|}(x)\leq p_{|y|}(y)=p(|y|)=p(y)$.\end{center}
This completes the proof.
\end{proof}

\section{Some notations and terminolgy} Let $E$ be a Riesz space. A subset $A\subset E$ is called {\bf\emph{solid}} if $y\in A$ whenever $|y|\leq |x|$ in $E$ for some $x\in A$. A linear topology  ${\tau}$ on $E$ is called {\bf\emph{locally solid}} if it has a neighborhood system at zero consisting of solid sets. One can easily  show that a given set $P$ of Riesz pseudonorms has a solid topology with a subbase of zero as the set  $\{p^{-1}(-{\epsilon},{\epsilon}):p\in P, {\epsilon}>0\}$. This topology is denoted by $<P>$, and it is called {\bf\emph{locally solid topology generated by $P$}}. Conversely, Fremlins Theorem says that every locally solid topology is  generated by a family of Riesz pseudonorms. That is, a linear topology ${\tau}$ is locally solid if and only if ${\tau}=<P>$ for some set $P$ of Riesz pseudonorms.(see \cite{frem}) .\\ 
 $\indent$ Let $p$ be a Riesz pseudonorm on $E$. For each $u \in E _{+}$, the map $p_{u}: E \rightarrow \mathbb{R}$ is also a Riesz pseudonorm  defined by $p_{u}(x)=p(|x| \wedge u)$. Let $(E,\tau)$ be a locally solid Riesz space. So there exists a family of Riesz pseudonorms $(p_{i})_{i \in I}$ such that $ \tau = < (p_{i})_{i \in I} > $. For any $A \subset E_{+}$, there exists a different family of Riesz pseudonorms $(p_{i,a})_{i \in I, a\in A}$ where $p_{i,a}(x)=p(|x| \wedge a)$ for each $i \in I$ and $a \in A$. This related family defines a locally solid topology. This fact coincides with the Mitchell A. Taylor's definition of "unbounded $\tau-$convergence with respect to $A$" in \cite{taylor}. Here is the Mitchell A. Taylor's definition.
\begin{definition} Let $X$ be a vector lattice, $A \subseteq X$ be an ideal and $\tau$ be a locally solid topology on $A$. Let $(x_{\alpha})$ be a net in $X$ and $x\in X$. We say that $(x_{\alpha})$ unbounded $\tau$-converges to $x$ with respect to
$A$ if $|(x_{\alpha})-x|\wedge |a| \xrightarrow{\tau}$ for all $a \in A_{+}$.
\end{definition}
$\indent$ In \cite{taylor}, the topology corresponding to the convergence in the above definition is denoted by $u_{A}\tau$.
\subsection*{Observation:} Let $E$ be a Riesz space, and $p: E \rightarrow \mathbb{R}$ be a Riesz pseudonorm. For a given nonempty set $A \subset E_{+}$, consider the map $p_{A}: E \rightarrow \mathbb{R}$ defined by \begin{center}
$p_{A}(x)= \sup_{a \in A}p(|x|\wedge a)$ 
\end{center} 
It is obvious that the map $p_{A}$ satisfies the conditions $(1)-(3)$ and $(5)$, we must check condition $(4)$: Let $\{ \lambda_{n}\} \subset \mathbb{R}$ be any sequence converging to zero. Then \begin{center}
$p_{A}(\lambda_{n}x)=\sup_{a \in A}p(|\lambda_{n} x| \wedge a)=\sup_{a \in A}p(|\lambda_{n}| | x| \wedge a)$
\end{center} 
\begin{center}
$\hspace{4.7 cm}$ $ \leq \sup_{a \in A}p(|\lambda_{n}||x|)$
\end{center}
\begin{center}   
$\hspace{4.5 cm}$ $ =p(|\lambda_{n}||x|)\longrightarrow 0$
\end{center}
so that $p_{A}$ is a Riesz pseudonorm. \\ $\indent$
Let $P=(p_{i})_{i \in I}$ be a family of Riesz pseudonorms and $\mathscr{A} \subset \mathscr{P}(E_{+})$ that does not contain the empty set. This family generates a topology, say $\tau$. The locally solid topology generated by the family $\{ p_{i,A}:i \in I, A \in \mathscr{A}\}$ will be denoted by $u<\tau,\mathscr{A}>$. Actually, if $\mathscr{A}$ contains the empty set, then $u<\tau,\mathscr{A}>$  is nothing but a discrete topology.
\subsection*{Some remarks:} Let $(E,\tau)$ be a locally solid Riesz space, $(p_{i})_{i \in I}$ be the family of Riesz pseudonorms such that $\tau = <(p_{i})_{i \in I}>$. Then
\begin{enumerate}
\item[(1)] For any $\mathscr{A} \subset \mathscr{P}(E_{+}) $, $u<\tau,\mathscr{A}> \subset \tau$ holds.\\
Proof: $x_{\alpha} \xrightarrow{\tau} x \Longleftrightarrow p_{i}(x_{\alpha} -x)\rightarrow 0 \Longleftrightarrow p_{i}(|x_{\alpha} -x|)\rightarrow 0$, and for each $a \in E_{+}$ we have $p_{i}(|x_{\alpha} -x|\wedge a) \leq p_{i}(|x_{\alpha} -x|) $, hence \begin{center}
$\sup_{a \in A}p_{i}(|x_{\alpha} -x|\wedge a) \leq p_{i}(|x_{\alpha} -x|)$ for each $A \in \mathscr{A}$.
\end{center}
So $x_{\alpha}\xrightarrow{u<\tau,\mathscr{A}>}x$.
\item[(2)]If $\mathscr{A}=\{\{E_{+}\}\}$, then $u<\tau,\{\{E_{+}\}\} >= \tau$. \\
Proof: It is clear that $\sup_{a \in E_{+}}p_{i}(|x|\wedge a)=p_{i}(|x|)=p_{i}(x)$.
\item[(3)] If $\mathscr{A} \subset \mathscr{B}$ then $u<\tau,\mathscr{A}> \subset u<\tau,\mathscr{B}>$ for all $\mathscr{A} \subset \mathscr{B} \subset \mathscr{P} (E_{+})$.\\
Proof: $x_{\alpha}\xrightarrow{u<\tau,\mathscr{B}>}x \Longleftrightarrow \sup_{b \in B}p_{i}(|x_{\alpha} -x|\wedge a)\rightarrow 0$ for each  $B \in \mathscr{B}$ \begin{center}
 $\hspace{2.99 cm}$ $\Longrightarrow \sup_{a \in A}p_{i}(|x_{\alpha} -x|\wedge a)\rightarrow 0$ for each $A \in \mathscr{A} \subset \mathscr{B}$.
\end{center} 
Hence, $x_{\alpha}\xrightarrow{u<\tau,\mathscr{A}>}x$.
\item[(4)] For each $\mathscr{A} \subset \mathscr{P} (E_{+})$ $u<\tau,\mathscr{\bigcup A}> \subset u<\tau,\mathscr{A}>$ holds. \\
Proof: Let $x_{\alpha}\xrightarrow{u<\tau,\mathscr{A}>}x$. So for a fixed $i \in I$ and $A \in \mathscr{A}$,  $p_{i,A}(x_{\alpha}-x)\rightarrow 0 \Longleftrightarrow \sup_{a\in A} p_{i}(|x_{\alpha}-x| \wedge a) \rightarrow 0 $, and it is obvious that \begin{center}
$p_{i}(|x_{\alpha}-x|\wedge a) \leq \sup_{a\in A} p_{i}(|x_{\alpha}-x| \wedge a) $ for each $a \in A$.
\end{center}
Hence, $p_{i,\{a \}}(x_{\alpha}-x)=p_{i,\{a \}}(|x_{\alpha}-x|)=\sup_{a\in \{a\}} p_{i}(|x_{\alpha}-x| \wedge a)=p_{i}(|x_{\alpha}-x| \wedge a) \rightarrow 0 $.
\item[(5)] For each $\mathscr{A} \subset \mathscr{P}(E_{+})$  $u<\tau,\mathscr{\bigcup A}>=u<\tau,I(\mathscr{\bigcup A})>$ holds where $I( \mathscr{\bigcup A})$ is the ideal generated by $\mathscr{\bigcup A}$. \\ Proof: Since $\mathscr{\bigcup A} \subset I(\mathscr{\bigcup A}) $, we have $u<\tau,\mathscr{\bigcup A}> \subset u<\tau,I(\mathscr{\bigcup A})>$ from $(3)$. Let  $x_{\alpha}\xrightarrow{u<\tau,\mathscr{A}>}x$ and $b \in I(\mathscr{\bigcup A})_{+} $ be given, there exists $a_{1},...,a_{n}\in \mathscr{\bigcup A}$ and $k \geq 0$ such that $0\leq b\leq k(a_{1}+...+a_{n})$. Then
\begin{center}$|x_{\alpha}-x|\wedge b\leq |x_{\alpha}-x|\wedge k(a_{1}+...+a_{n})\leq \sum_{i=1}^{n}|x_{\alpha}-x|\wedge ka_{i}$,\end{center}
\begin{center}
$\hspace{7.4 cm}$ $=k\sum_{i=1}^{n} \frac{1}{k}|x_{\alpha}-x|\wedge a_{i}$
\end{center}
\begin{center}
$\hspace{7.5 cm}$ $\leq km\sum_{i=1}^{n}|x_{\alpha}-x|\wedge a_{i}$
\end{center}
where $m$ is the smallest positive integer  greater than $\frac{1}{k}$. Then by the monotonicity of $p_{i}$,
\begin{center}
$p_{i}(|x_{\alpha}-x|\wedge b) \leq p_{i}(km\sum_{i=1}^{n}|x_{\alpha}-x|\wedge a_{i}) \rightarrow 0 $.
\end{center}
Hence, $p_{i}(|x_{\alpha}-x|\wedge b)= \sup_{b \in \{b\}} p_{i}(|x_{\alpha}-x|\wedge b)$. This completes the proof. 
\item[(6)]For each $\mathscr{A} \subset \mathscr{P} (E_{+})$  $u<\tau,\mathscr{\bigcup A}>=u<\tau,\mathscr{ \overline{\bigcup A}}>$ holds. \\ 
Proof: Suppose that $x_{\alpha} \xrightarrow{u<\tau,\mathscr{A}>}x$ and $b \in (\mathscr{\overline{\bigcup A}})_{+}$ be given. Choose a net $(b_{\beta}) \in \mathscr{\bigcup A}$ with $b_{\beta} \xrightarrow{u<\tau,\mathscr{A}>}b$. Let $i \in I$ be fixed and $\epsilon \geq 0$ be given. Choose $\beta_{0}$ such that $p_{i}(b_{\beta_{0}}-b)< \frac{\epsilon}{2}$. Then
\begin{center}
$|x_{\alpha}-x|\wedge b= |x_{\alpha}-x|\wedge (b-b_{{\beta}_{0}}+b_{{\beta}_{0}})$
\end{center}
\begin{center}$\hspace{2.5 cm}$ $\leq |x_{\alpha}-x|\wedge (|b-b_{{\beta}_{0}}|+\vert b_{{\beta}_{0}}\vert)$\end{center} \begin{center} $\hspace{4.2 cm}$
$\leq |x_{\alpha}-x|\wedge |b-b_{{\beta}_{0}}|+|x_{\alpha}-x|\wedge |b_{{\beta}_{0}}|$.
\end{center} 
Applying $p_{i}$ to this inequality,  one can show the existence of $\alpha_{0}$ such that $p_{i,\{b\}}(x_{\alpha}-x)< \epsilon$. This completes the proof. 
\item[(7)] If  $0 \leq a \leq b $, then $u<\tau,\{a\}>u<\tau,\{b\}>$. \\
Proof: It is clear that; \begin{center}$\sup_{a \in \{a\}}p_{i}(|x_{\alpha}-x|\wedge a)=p_{i}(|x_{\alpha}-x|\wedge a)$\end{center} \begin{center}
$\hspace{4.2 cm}$
$\leq p_{i}(|x_{\alpha}-x|\wedge b)$
\end{center} 
\begin{center}
$\hspace{5.8 cm}$
 $=\sup_{b \in \{b\}}p_{i}(|x_{\alpha}-x|\wedge b) $.
\end{center}
 
\item[(8)] If $e \in E$ is a strong order unit, then $u<\tau,\{\{e\}\}>=u<\tau,\bigcup E_{+}>$. But the converse of this statement is not true in general. For example, consider $c_{0}$  as a Banach lattice with supremum norm, with norm topology ${\tau}$ and 
$e=({1\over n})$. Then $u<{\tau},\{\{e\}\}>=u<{\tau,\bigcup E_{+}}>$, but $e$ is not an order unit.
\item[(9)] If $e$ is a quasi-interior point, then $u<{\tau},\{\{e\}\}>=u<{\tau,\bigcup E_{+}}>$ from $(5)$ and $(6)$.
\item[(10)]For any $\mathscr{A} \subset \mathscr{P}(E_{+}) $,  $u<\tau,\mathscr{A}>=u<u<\tau,\mathscr{A}>,\mathscr{A}>$ holds.
\item[(11)] For any $\{A\} \in \mathscr{P}(E_{+}) $, $u<\tau,\{A\}> \neq u_{A}\tau$, but $u<\tau,\bigcup A> = u_{A}\tau$ holds. Moreover, $ u<\tau,\{A\}> \subset u_{A}\tau $. As an example let consider $\mathbb{R}^2$ with Euclidean norm,and take the set of non-negative part of $x-$axis as $A$, then  the sequence $(x_{n})$ $(x_{n}:=2+\sin n)$ does not converges in $u_{A}\tau$,but converges in $u<\tau,\{A\}>$.  

\end{enumerate}

\section{Unbounded locally solid Riesz space}
From the motiviation of the above observation, we give the following definition.
\begin{definition} A real valued map $q$ on a Riesz space $E$ is said to be {\bf\emph{unbounded Riesz pseudonorm}} if there exists a Riesz pseudonorm $p$ on $E$ and  $ A \subset E^{+}$ satisfying $q(x)=\sup_{a \in A} p(|x|\wedge a)$. In this case, we say that $q$ is generated by $p$ and the subset $A$.
\end{definition}

It is obvious that every unbounded Riesz pseudonorm is a  Riesz pseudonorm. So the topology generated by unbounded Riesz pseudonorm is a locally solid topology. 
If unbounded Riesz pseudonorm $q$ is generated by  Riesz pseudonorm $p$ and $A\subset E^{+}$,  then the topology generated by $q$ is weaker than the topology generated by $p$. Remind that every family of  Riesz pseudonorms defines a locally solid topology. Conversely, every locally solid topology is determined by a family of Riesz pseudonorms. 
\begin{definition} Let $(E,{\tau})$ be a locally solid Riesz space generated by the family $(p_{i})_{i\in I}$ of Riesz pseudonorms . The locally solid Riesz space on $E$ generated by the family of unbounded Riesz pseudonorm on $E$ is called {\bf\emph{unbounded locally solid Riesz space}} generated by $\tau$, and denoted by $\tau^{'}$
\end{definition}
\begin{prop} Let $(E,{\tau})$ be a locally solid Riesz space. If $\tau$ is a Hausdorff locally solid topology, then the unbounded locally solid topology is also Hausdorff. 
\end{prop}
\begin{proof}
Let $(p_{i})_{i \in I}$ be a family of Riesz pseudonorms such that $\tau=<(p_{i})_{i \in I}>$ and $x \neq 0$ be given, then there exists some $i_{0} \in I$ such that $p_{i_{0}}(x)>0$.Then, \begin{center}
$q_{i_{0},\{|x|\}}:=\sup_{a\in \{|x|\}}p_{i_{0}}(|x|\wedge a)=p_{i_{0}}(|x|\wedge|x|)=p_{i_{0}}(|x|)=p_{i_{0}}(x)>0$.
\end{center}
It is obvious that $q_{i_{0},\{|x|\}}$ is an unbounded Riesz pseudonorm, so $\tau^{'}$ is a Hausdorff topology.
\end{proof}
\begin{definition} A net $(x_{\alpha})$ in a locally solid Riesz space $(E,\tau)$ is {\bf\emph{unbounded topological convergent }} if it is convergent in unbounded locally solid Riesz space $(E,\tau^{'})$.
\end{definition}
\begin{theorem}
Let $(E,{\tau})$ be a Hausdorff locally solid Riesz space and $(x_{\alpha})$ be an increasing net. Then the followings are equivalent:
\begin{enumerate}
\item $(x_{\alpha})\xrightarrow{\tau} x$ in $(E,\tau)$;
\item $(x_{\alpha})\xrightarrow{\tau^{'}}x$ in $(E,\tau^{'})$.
\end{enumerate}
\end{theorem}
\begin{proof}
Since $\tau^{'} \subset \tau$, it is easy to see that $(1)$ implies $(2)$. Now suppose $(2)$ holds. Since $\tau^{'}$ is a Hausdorff locally solid Riesz space by the Proposition 4.3, we have $x_{\alpha}\uparrow x$. Thus $|x|$ is an upper bound for the net $(x_{\alpha})$ and $2|x|$ is an upper bound for the net $(|x_{\alpha}-x|)$. Now suppose that  $(p_{i})_{i \in I}$ is the family of Riesz pseudonorms such that  $\tau=<(p_{i})_{i \in I}>$. Let $i \in I$ be arbitrary. Then, \begin{center}
$p_{i}(x_{\alpha}-x)=p_{i}(|x_{\alpha}-x|)=p_{i}(|x_{\alpha}-x|\wedge 2|x|)=\sup_{a\in \{2|x|\}}p_{i}(|x_{\alpha}-x|\wedge a)$
\end{center}
\begin{center} $\hspace{7.2 cm}$
$:=q_{i,\{2|x|\}}(x_{\alpha}-x)\rightarrow 0$.
\end{center}
This completes the proof.
\end{proof}
\begin{theorem}
Let $(E,{\tau})$ be a Hausdorff locally solid Riesz space, and $\tau^{'}$ be the unbounded locally solid topology generated by $\tau$. Then $\tau$ has Lebesgue property if and only if $\tau^{'}$ has Lebesgue property. 
\end{theorem}
\begin{proof}
One side of the implication is clear. Let us assume that $x_{\alpha}\downarrow 0$ implies $x_{\alpha} \xrightarrow{\tau^{'} }0$. Then , it is easy to see that $x_{\alpha} \xrightarrow{\tau} 0$ by using the Theorem 4.5. This completes the proof.  
\end{proof}

\subsection{Product of unbounded locally solid Riesz space}
\begin{theorem} Let $(E_{i},{\tau}_{i})_{i\in I}$ be a family of locally solid Riesz spaces. Then the product space $\prod_{i\in I}E_{i}$ is unbounded locally solid Riesz space if and only if for each $i$, $E_{i}$ is an unbounded locally solid Riesz space.
\end{theorem}
\begin{proof} Suppose that for each $i\in I$, $(E_{i},\tau_{i})$ is an unbounded locally solid Riesz space, and ${\tau}_{i}$ is generated by a family $Q_{i}$ of the unbounded Riesz pseudonorms on $E_{i}$. So for each $q\in Q_{i}$, there exists a Riesz pseudonorm $p$ on $E_{i}$ and $A_{i} \subset E_{i}^{+}$, depending on $q$, such that \begin{center}
 $q(x)=\sup_{a\in A_{i}}p(|x|\wedge a)$ for all $x\in E_{i}$. \end{center} Let $j \in I$ and $q\in Q_{j}$ be given. Choose $p$ and $A_{j}$ as above. Let $P_{j}$ be the projection from $E=\prod_{i}E_{i}$ into $E_{j}$ and $f_{j}$ be vector space embedding of $E_{j}$ into $E$, that is, $f_{j}$ sends $x\in E_{j}$ to $(x_{i})$ where $x_{j}=x$ and $x_{i}=0$ for all $i\not =j$. One can show that for each Riesz pseudonorm  on $E_{j}$, $p\circ P_{j}$ is a Riesz pseudonorm on $E$. We note that for each $q\in Q_{j}$,
\begin{center}$q\circ P_{j}((x_{i}))=q(P_{j}(x_{i}))=q(x_{j})=\sup_{a\in A}p(|x_{j}|\wedge a)=sup_{a\in A_{j}} \hspace{0.1 cm} p\circ P_{j}(|(x_{i})|\wedge f_{j}(a))$.\end{center} Thus, $q\circ P_{j}$ is an unbounded Riesz pseudonorm on $E$. And the the topology of $\prod_{i}E_{i}$ is the topology generated by $\{q\circ P_{j}:j\in I, q\in Q_{j}\}$. Hence, the locally solid Riesz space $\prod_{i}E_{i}$ is an unbounded locally solid Riesz space.

Now suppose that $E=\prod_{i}E_{i}$ is an unbounded locally solid Riesz space, and $i_{0}$ is given. Suppose that the topology of $E$ is generated by the family $Q$ of unbounded Riesz pseudonorm on $E$. Let $q\in Q$ be given. There exists $A=(A_{i})\in E_{+}$ and Riesz pseudonorm $p$ on $E$ such that 
$q(x)=\sup_{a\in A}p(|x|\wedge a)$ for all $x\in E$. It is obvious that for each $i_{0}$, $p\circ f_{i_{0}}$ is a Riesz pseudonorm on $E_{i_{0}}$ and 
\begin{center}$q\circ f_{i_{0}}(x)=\sup_{a\in A_{i_{0}}}p\circ f_{i_{0}}(|x|\wedge a)$.\end{center} Hence $q_{i_{0}}$ is an unbounded Riesz pseudonorm on $E_{i_{0}}$. Now one can show that the topology of $E_{i_{0}}$ is generated by $\{q\circ f_{i_{0}}:q\in Q\}$. Hence, $E_{i_{0}}$ is an unbounded locally solid Riesz space. This completes the proof.
\end{proof}
Let $X$ be a product space of topological spaces $(X_{i})_{i\in I}$. A net $(x_{\alpha})$ converges to $x$ in $X$ if and only if  $x_{\alpha}^{i}\rightarrow x_{i}$ in $X_{i}$  
for each $i\in I$, where $x_{\alpha}=(x_{\alpha}^{i})_{i\in I}$ and $x=(x_{i})$. By using this fact, the proof of the following theorem is easy.
\begin{theorem} Let  $(E_{i},{\tau}_{i})_{i\in I}$ be a family of locally solid Riesz spaces. For each $\mathscr{A}_{i}\subset \mathscr{P}(E_{i}^{+})$, we have
\begin{center}$u<{\prod}_{i}{\tau}_{i},{\prod}_{i}\mathscr{A}_{i}>={\prod}_{i}u<{\tau}_{i},\mathscr{A}_{i}>$.\end{center}
\end{theorem}

\subsection{Unbounded absolute weakly locally solid Riesz space} 

The concept of unbounded absolute weak convergence (briefly uaw-convergence) was considered and studied in \cite{omd}. Let $E$ and $F$ be vector spaces. If there exists a bilenear map $T:E\times F\rightarrow\mathbb{R}$ satisfying

$T(x,F)=0\Longrightarrow x=0$, 

$T(E,y)=0\Longrightarrow x=0$, 
then the pair $(E,F)$ is called a dual pair. In this case, $E$ can be considered as a vector supspace of $\mathbb{R}^{F}$, by embedding $x\rightarrow x^{*}$, $x^{*}(y)=T(x,y)$. We can consider ${\mathbb{R}}^{F}$ as a topological space with product topology $\prod_{y\in F}\mathbb{R}$ and restriction of this topology on $E$ is the topology generated by the family $(p_{y})_{y\in F})$ of seminorms, where $p_{y}:E\rightarrow\mathbb{R}$ defined by $p_{y}(x)=|T(x,y)|$. This topology is independent of $T$ and is denoted by ${\sigma}(E,F)$. Similarly, ${\sigma}(F,E)$ can be defined. One of the main results is that the topological dual of $E$ with respect to ${\sigma}(E,F)$ is a vector space  which is isomorphic to $F$, this is denoted by $(E,{\sigma}(E,F))^{'}\cong F$.
\begin{definition} If $(E,F)$ is  a dual pair of Riesz spaces with respect to a positive linear map $T:E\times F\rightarrow\mathbb{R}$, then we call that as a {\bf\emph{positive dual pair}} (with respect to $T$). 
\end{definition}
We note that if $(E,F)$ is a positive dual pair with bilenear map $T$, then one can show that the embedding $x\rightarrow x^{*}$, $x^{*}(y)=T(x,y)$ is bipositive. 

The order dual of a Riesz space $E$ is  the vector space of order bounded functionals from $E$ into $\mathbb{R}$ and denoted by $E^{\sim}$, which is a Dedekind complete Riesz space. Throughout the paper we suppose that $E$ separates its order dual, that is, for each nonzero $x\in E$, there exists $f\in E^{\sim}$ with $f(x)\not =0$.  
So, $(E,E^{\sim})$ is a positive dual pair via the map $(x,f)\rightarrow f(x)$. If ${\tau}$ is a Hausdorff locally solid topology on $E$, then the topological dual $E^{'}$ is an ideal of $E^{\sim}$. Let $A\subset E^{\sim}$ be given. For each $f\in A$, the map $p_{|f|} : E\rightarrow\mathbb{R}$. $p_{|f|}(x)=|f|(|x|)$ is a Riesz seminorm. The locally convex-solid topology generated by $(p_{|f|})_{f\in E^{\sim}}$ is called {\bf\emph{absolute weak topology}} and denoted by $|{\sigma}|(E,A)$.

Now we are going to define an unbounded absolute locally solid topology. For this, first we need the following Lemma.
\begin{lemma} Let $(E,F)$ be a positive dual pair with respect to $T$. For each $a\in E$ and $y\in F$, the map $p:E\rightarrow\mathbb{R}$ defined by
\begin{center}$p(x)=T(|x|\wedge |a|,|y|)$\end{center} is a Riesz pseudonorm on $E$.
\end{lemma}
\begin{proof} Without loss of the generality, we can suppose that $a$ and $y$ are positive. Obviously the conditions $(1),(2)$ and $(5)$ are satisfied . For the condition $(3)$: for a given pair $x,y \in E$, $p(x+y)=T(|x+y|\wedge a,y)$ \begin{center}
$\hspace*{-1.32 cm}$
$\leq T(|x|+|y|\wedge a,y)$ by positivity 
\end{center}  
\begin{center}
$\hspace*{3.3 cm}$
$\leq T(|x|\wedge a,y)+T(|y|\wedge a,y)$ by bilinearity and positivity
\end{center}
\begin{center}
$\hspace*{-4.64 cm}$
$=p(x)+p(y)$,
\end{center}
hence, the condition $(3)$ holds. For the condition $(4)$, let $\{{\lambda}_{n}\} \subset\mathbb{R}$ be a sequence such that $\lim_{n \rightarrow \infty} \lambda_{n}=0 $ and $x\in E$, we have \begin{center}
$p(\lambda_{n}x)=T(|\lambda_{n}x|\wedge a,y)=T(|{\lambda_{n}}||x|\wedge a,y)=T(|\lambda_{n}|(|x|\wedge {1\over {|\lambda_{n}|}}a),y)$\end{center}\begin{center}$\hspace{7.6 cm}$ $=|{\lambda_{n}}|T(|x|\wedge {1\over {|\lambda_{n}|}}a),y)$\end{center} \begin{center} $\hspace{6.3 cm}$ $\leq  |{\lambda_{n}}|T(|x|,y)$
\end{center}
So, $T(|x|,y)$is a real number,$|{\lambda_{n}}|T(|x|,y)\rightarrow 0$,thus the condition $(4)$also holds.
\end{proof}
By using the same motivation, for a given $A\subset E_{+},e_{0}\in E$ and $f_{0}\in F$, the map $\sup_{a\in A}T(|x|\wedge |e_{0}|,|f_{0}|)$ is also a Riesz Pseudonorm, and it will be denoted by $p_{A,e_{0},f_{0}}$ 
\begin{definition} Let $(E,F)$ be a positive dual pair. Let $E_{0}\subset E$,$F_{0}\subset F$ and $\mathscr{A}\subset \mathscr{P}(E_{+})$ be nonempty sets. Then the topology generated by 
$(p_{A,e_{0},f_{0}})_{{A\in \mathscr{A},e_{0}\in E_{0},f_{0}\in F_{0}}}$ is called unbounded  locally solid Riesz space on the positive pair $(E,F)$ with respect to $E_{0},F_{0}$ and $\mathscr{A}$.
This topology is denoted by $u|{\sigma}|(E,F),E_{0},F_{0},\mathscr{A})$.
\end{definition}
By using some routine arguments, the proof of the above theorem can be given.
\begin{theorem}Let $(E,F)$ be a positive dual pair. Let Let $E_{0}\subset E$,$F_{0}\subset F$ and $\mathscr{A}\subset \mathscr{P}(E_{+})$ be nonempty sets. Then 
\begin{center}$u|{\sigma}|(E,F),E_{0},F_{0},\mathscr{A})=u|{\sigma}|(E,F),I(E_{0}),I(F_{0}),\mathscr{A})$.\end{center}
\end{theorem}
\begin{rem} These observations and results can be extended  into locally solid lattice-ordered groups studied in \cite{hong}.
\end{rem}

\end{document}